\documentclass[10pt]{article}

\usepackage{amsmath,amsthm,amssymb,bbm}
\usepackage[bookmarksnumbered=true,pagebackref]{hyperref}
\usepackage{tocbibind}
\usepackage{enumitem}

\newtheorem{thm}{Theorem}

\newtheorem{lem}{Lemma}

\theoremstyle{remark}
\newtheorem{rem}{Remark}

\theoremstyle{definition}

\newtheoremstyle{notes}% name
{3pt}% Space above
{3pt}% Space below
{}% Body font
{}% Indent amount
{\bfseries}% Theorem head font
{:}% Punctuation after theorem head
{.5em}% Space after theorem head
{}% Theorem head spec (can be left empty, meaning normal)
\theoremstyle{notes}
\newtheorem*{keywords}{Keywords}
\newtheorem*{subjclass}{AMS MSC 2010}
\newtheorem*{acknowledgments}{Acknowledgments}

\newcommand{\D}[1]{\mathop{\mathrm{d}#1}}

\usepackage{comment}
\usepackage[draft]{todonotes}
\title{Independent factorization of the last zero\\ arcsine law for Bessel processes with drift}

\author{Hugo Panzo\footnote{Supported at the Technion by a Zuckerman Fellowship.} \\ 
\href{mailto:panzo@campus.technion.ac.il}{\texttt{{\small panzo@campus.technion.ac.il}}}}

\date{\today}

\begin{document}

\maketitle

\begin{abstract}
We show that the last zero before time $t$ of a recurrent Bessel process with drift starting at $0$ has the same distribution as the product of an independent right censored exponential random variable and a beta random variable. This extends a recent result of Schulte-Geers and Stadje \cite{ESG} from Brownian motion with drift to recurrent Bessel processes with drift. Our proof is intuitive and direct while avoiding heavy computations. For this we develop a novel additive decomposition for the square of a Bessel process with drift that may be of independent interest.
\end{abstract}

\begin{keywords}
Bessel processes with drift, arcsine law, time inversion property, additivity property, last exit time.
\end{keywords}

\begin{subjclass}
Primary 60G17, 60J60; Secondary 60J65.
\end{subjclass}

\section{Introduction}\label{sec:intro}

Let $X=(X_t:t\geq 0)$ denote the coordinate process on the canonical space of continuous functions from $[0,\infty)$ to $\mathbb{R}$ and let $\mathbb{W}_x$ denote the law under which $X$ is Brownian motion starting at $x\in\mathbb{R}$. We write $g_t$ for the last zero of $X$ before time $t$. More precisely, $g_t=\sup\{s\leq t:X_s=0\}$ with the usual convention of $\sup\{\emptyset\}=-\infty$. The well-known arcsine law for $g_t$ is due to L\'{e}vy \cite[\S~7.7]{Levy} and states that
\begin{equation}\label{eq:arcsine}
g_t\text{ under }\mathbb{W}_0\,\stackrel{\mathcal{L}}{=}\, tA_\frac{1}{2}
\end{equation}
where we used $\stackrel{\mathcal{L}}{=}$ to indicate equality in law and $A_\frac{1}{2}$ is a $\mathrm{Beta}(\frac{1}{2},\frac{1}{2})$ random variable with density
\[
P(A_\frac{1}{2}\in\D{x})=\frac{1}{\pi\sqrt{x(1-x)}}\mathbbm{1}_{(0,1)}(x)\D{x}.
\]

There have been many extensions of \eqref{eq:arcsine} to processes besides $1$-dimensional Brownian motion, see Chapter 8 of \cite{aspects} and Section 2.5 of \cite{Pitman_arcsine} as well as references therein for some examples. One such generalization that plays an important role in this paper is due to Lamperti \cite[Theorem 5.1]{Lamperti_arcsine} and identifies the distribution of $g_t$ when the underlying Brownian motion is replaced by a recurrent Bessel process of dimension $\delta\in (0,2)$. We use $\mathbb{P}_x^\delta$ to denote the law of this Bessel processes when it starts at $x\geq 0$. Lamperti showed that 
\begin{equation}\label{eq:Bessel_arcsine}
g_t\text{ under }\mathbb{P}^\delta_0\,\stackrel{\mathcal{L}}{=}\, t A_\alpha
\end{equation}
where $\alpha=1-\frac{\delta}{2}\in (0,1)$ and $A_\alpha$ is a $\mathrm{Beta}(\alpha,1-\alpha)$ random variable with density
\[
P(A_\alpha\in\D{x})=\frac{x^{\alpha-1}(1-x)^{-\alpha}}{\Gamma(\alpha)\Gamma(1-\alpha)}\mathbbm{1}_{(0,1)}(x)\D{x}.
\]
\begin{comment}
Dynkin '61 shows that the undershoot of a stable subordinator has the generalized arcsine law. To get the last zero of a recurrent Bessel process from this you need Molchanov-Ostrovski '69.
\end{comment}
L\'{e}vy's last zero arcsine law \eqref{eq:arcsine} can be seen as a special case of \eqref{eq:Bessel_arcsine} since the law of $|X|$ under $\mathbb{W}_0$ is the same as the law of $X$ under $\mathbb{P}_0^1$.

Another more recent generalization that is central to our results is that of characterizing the law of $g_t$ when a constant drift $\mu\in\mathbb{R}$ is added to the underlying Brownian motion. We use $\mathbb{W}_x^\mu$ to denote the law of Brownian motion with constant drift $\mu$ when it starts at $x\in\mathbb{R}$. Using a random walk approximation argument, Schulte-Geers and Stadje \cite[Theorem 2.1]{ESG} show that $g_t$ under $\mathbb{W}_0^\mu$ has the independent factorization
\begin{equation}\label{eq:ESG}
g_t\text{ under }\mathbb{W}^\mu_0\,\stackrel{\mathcal{L}}{=}\, \min\{t,E_\mu\}\, A_\frac{1}{2}
\end{equation}
where $E_\mu$ is an $\mathrm{Exp}(\frac{1}{2}\mu^2)$ random variable independent of $A_\frac{1}{2}$ and with density
\[
P(E_\mu\in\D{t})=\frac{1}{2}\mu^2 e^{-\frac{1}{2}\mu^2 t}\mathbbm{1}_{[0,\infty)}(t)\D{t}.
\] 

Despite the elegant and elementary nature of \eqref{eq:ESG}, it seems to have escaped notice until \cite{ESG}, see also \cite[Remark 2.1]{Iafrate}. The present author first learned of this characterization from a MathOverflow answer \cite{MathOverflow}, presumably written by the first author of \cite{ESG}. Another attractive feature of \eqref{eq:ESG} is that it allows us to easily recover L\'{e}vy's last zero arcsine law \eqref{eq:arcsine} since $\min\{t,E_\mu\}$ degenerates to $t$ as $\mu\to 0$. Moreover, since the last exit time from $0$ is almost surely finite when $\mu\neq 0$, we can also recover the law of $g_\infty$ in this case. Letting $t\to\infty$ in \eqref{eq:ESG} shows that $g_\infty$ under $\mathbb{W}_0^\mu$ is distributed like $E_\mu\,A_\frac{1}{2}$. After recalling the beta-gamma algebra \cite[Chapter 0.6]{Revuz_Yor}, it follows that $g_\infty$ under $\mathbb{W}_0^\mu$ is a $\mathrm{Gamma}(\frac{1}{2},\frac{1}{2}\mu^2)$ random variable with density 
\[
\mathbb{W}_0^\mu(g_\infty\in\D{t})=\frac{|\mu|}{\sqrt{2\pi t}}e^{-\frac{1}{2}\mu^2 t}\mathbbm{1}_{(0,\infty)}(t)\D{t}.
\]

While verifying \eqref{eq:ESG} directly using a Girsanov measure change argument is a straightforward matter, tedious calculations are required as can be witnessed in Section 2 of \cite{Iafrate} where this is carried out in detail. Indeed, in \cite[Remark 2.3]{ESG}, the authors appeal for a ``purely Brownian'' explanation of the independent factorization \eqref{eq:ESG}. This leads us to the main contributions of this paper:
\begin{enumerate}
\item Extending \eqref{eq:ESG} to Bessel processes of dimension $\delta\in(0,2)$ with positive drift (in the sense of Watanabe \cite{wide_sense}), thereby unifying the last zero arcsine laws for Brownian motion \eqref{eq:arcsine} and recurrent Bessel processes \eqref{eq:Bessel_arcsine} under the independent factorization framework of \eqref{eq:ESG} when drift is present. 
\item Giving a ``purely Bessel'' explanation for the independent factorization \eqref{eq:ESG} and the aforementioned extension to Bessel processes with drift that is intuitive, direct, and avoids heavy computation. For this we develop an additive decomposition for the square of a Bessel process with drift which appears to be new and may be of independent interest.
\end{enumerate}

The remainder of the paper is organized as follows. In Section \ref{sec:main_results} we recall the definition of Bessel processes with drift and state our main theorems. In Section \ref{sec:preliminaries}, we review several relevant properties of Bessel processes and bridges and prove some preliminary results. Finally, the main theorems are proved in Section \ref{sec:proofs} where we also outline an alternative proof in Section \ref{sec:last_zero} and discuss how it compares to the proof we give.

\section{Main results}\label{sec:main_results}

\subsection{Bessel processes with drift}

Before stating our main results, we first recall the \emph{Bessel processes with drift} introduced by Watanabe \cite{wide_sense}. These comprise a two-parameter family of diffusions on $[0,\infty)$ that are indexed by their dimension $\delta>0$ and drift $\mu\geq 0$. Let $I_\nu$ denote the modified Bessel function of the first kind and for $\mu>0$ define the function $h_{\delta,\mu}:[0,\infty)\to[1,\infty)$ by
\begin{equation*}
h_{\delta,\mu}(x)=\begin{cases}
1&x=0\\
\left(\frac{2}{\mu\, x}\right)^{\frac{\delta}{2}-1}\Gamma\left(\frac{\delta}{2}\right)I_{\frac{\delta}{2}-1}(\mu\, x)&x>0.
\end{cases}
\end{equation*}
When $\mu>0$, these processes are determined by the generator
\begin{equation}\label{eq:generator}
L^{\delta,\mu}=\frac{1}{2}\frac{\D{}^2}{\D{x}^2}+\left(\frac{\delta-1}{2x}+\frac{h'_{\delta,\mu}(x)}{h_{\delta,\mu}(x)}\right)\frac{\D{}}{\D{x}}
\end{equation}
with $0$ being a regular boundary with instantaneous reflection if $0<\delta<2$ or an entrance boundary if $\delta\geq 2$. When $\mu=0$, they coincide with the usual Bessel processes having the same dimension. Accordingly, we use $\mathbb{P}_x^{\delta,\mu}$ to denote the law of a Bessel process with dimension $\delta$ and drift $\mu$ that starts from $x\geq 0$. By writing $\mathbb{P}_x^\delta$ instead of $\mathbb{P}_x^{\delta,0}$, this notation subsumes that of Bessel processes without drift from Section \ref{sec:intro}. We will also work with \emph{squared Bessel processes with drift} and will use $\mathbb{Q}_x^{\delta,\mu}$ to denote the law of these processes. More precisely, 
\begin{equation}\label{eq:Bessel_square}
\left(X_t:t\geq 0;\,\mathbb{Q}_x^{\delta,\mu}\right)\stackrel{\mathcal{L}}{=}\left(X_t^2:t\geq 0;\,\mathbb{P}_{\sqrt{x}}^{\delta,\mu}\right).
\end{equation}

The appearance of the logarithmic derivative of $h_{\delta,\mu}$ in the first-order term of \eqref{eq:generator} along with the fact that $L^{\delta,0}\,h_{\delta,\mu}=\frac{1}{2}\mu^2\, h_{\delta,\mu}$ implies that a Bessel process of dimension $\delta$ with drift $\mu$ is simply a Bessel process of the same dimension without drift killed at rate $\frac{1}{2}\mu^2$ and then $h$-transformed by $h_{\delta,\mu}$. In particular, for any fixed $t\geq 0$, this gives the absolute continuity relation 
\begin{equation}\label{eq:h_transform}
\D{\mathbb{P}_x^{\delta,\mu}}\big|_{\mathcal{F}_t}=e^{-\frac{1}{2}\mu^2 t}\frac{h_{\delta,\mu}(X_t)}{{h_{\delta,\mu}(x)}}\D{\mathbb{P}_x^\delta}\big|_{\mathcal{F}_t}
\end{equation}
where we used $(\mathcal{F}_t:t\geq 0)$ to denote the canonical filtration. Refer to \cite[Section 4.1]{Pinsky} for the requisite theory on $h$-transforms of diffusion processes.

As remarked upon in \cite{infinite_divis}, the name Bessel process with drift is appropriate since for $\delta\in\mathbb{N}$, the law of $X$ under $\mathbb{P}_0^{\delta,\mu}$ is the same as that of the modulus of Brownian motion in $\mathbb{R}^\delta$ starting at $0$ with a constant drift vector of magnitude $\mu$, see \cite[Theorem 3]{Markov_functions}. Moreover, by the corollary to \cite[Theorem 2.1]{wide_sense}, we have for any $\delta>0$ and $x,\mu\geq0$
\begin{equation}\label{eq:limit_speed}
\mathbb{P}_x^{\delta,\mu}\left(\lim_{t\to\infty}\frac{X_t}{t}=\mu\right)=1.
\end{equation}

Interest in these processes is motivated in part by their being, up to a scale factor, the only regular and conservative diffusions on $[0,\infty)$ that satisfy the \emph{time inversion property} \cite{wide_sense, Lawi}, see Section \ref{sec:time_inversion}. Additionally, when $\delta=3$ they make an appearance in Williams' path decomposition for Brownian motion with drift \cite{Williams_decomp, Bessel_sup}. When $\delta=3$ and $\mu=1$, they also coincide with the hyperbolic Bessel process of dimension $3$, see \cite{hyper_Bessel, inversions}. We note that they are distinct from the Bessel processes with constant ``naive drift'' which arise in studies of bird navigation and queueing theory, see \cite{naive_drift,Linetsky} and references therein.

\subsection{Main theorems}

Our first result is an additive decomposition of $(X_t:t\geq 0)$ under $\mathbb{Q}_0^{\delta,\mu}$ into two independent processes that start at $0$: one being a squared Bessel process of dimension $\delta$ without drift and the other being a squared Bessel process of dimension $4$ with drift $\mu$ which waits for an independent $\mathrm{Exp}(\frac{1}{2}\mu^2)$ time before starting. Such a decomposition was alluded to by Pitman and Yor in their Remark 5.8.iii of \cite{Bessel_decomp} but as far as we know, no explicit statement has ever appeared in the literature. A random waiting time before starting has featured in a similar additive decomposition for a squared Bessel process of dimension $2$ without drift that appears in Section 3.5.1 of \cite{aspects}. We adopt Mansuy and Yor's notation which makes use of the positive part $x^+:=\max\{0,x\}$.
\begin{thm}\label{thm:decomp}
Suppose $\delta,\mu>0$ and let $E_\mu$ be an independent $\mathrm{Exp}(\frac{1}{2}\mu^2)$ random variable. Then we have
\[
\left(X_t:t\geq 0;\,\mathbb{Q}_0^{\delta,\mu}\right)\stackrel{\mathcal{L}}{=}\left(X_t+X_{(t-E_\mu)^+}':t\geq 0;\,\mathbb{Q}_0^\delta(X)\times\mathbb{Q}_0^{4,\mu}(X')\right)
\]
where the product law notation $\mathbb{Q}_0^\delta(X)\times\mathbb{Q}_0^{4,\mu}(X')$ indicates that $X$ and $X'$ are independent with laws $\mathbb{Q}_0^\delta$ and $\mathbb{Q}_0^{4,\mu}$, respectively.
\end{thm}
\begin{rem}
The piecewise nature of the second summand distinguishes the additivity exhibited in Theorem \ref{thm:decomp} from the usual kind described in Section \ref{sec:additivity}. A similarly exotic form of additivity for squared Bessel processes without drift and with possibly negative dimensions can be found in \cite[Proposition 1.1]{Pitman_squared}.
\end{rem}

Theorem \ref{thm:decomp} allows us to give a quick and intuitive proof of the independent factorization of the last zero arcsine law \eqref{eq:ESG} and its generalization to Bessel processes with drift which we state below as Theorem \ref{thm:last_zero}. 
\begin{thm}\label{thm:last_zero}
Let $g_t=\sup\{s\leq t:X_s=0\}$ be the last zero before time $t$ of a Bessel process with dimension $\delta\in(0,2)$ and drift $\mu>0$ starting at $0$. Put $\alpha=1-\frac{\delta}{2}$ and let $A_\alpha$ and $E_\mu$ be independent $\mathrm{Beta}(\alpha,1-\alpha)$ and $\mathrm{Exp}(\frac{1}{2}\mu^2)$ random variables, respectively. Then we have the independent factorization
\[
g_t\text{ under }\,\mathbb{P}^{\delta,\mu}_0\,\stackrel{\mathcal{L}}{=}\,\min\{t,E_\mu\}\, A_\alpha.
\]
\end{thm}
\begin{rem}
As mentioned before, $|X|$ under $\mathbb{W}_0^\mu$ has the same law as $X$ under $\mathbb{P}_0^{1,|\mu|}$ so it follows that \eqref{eq:ESG} is a special case of Theorem \ref{thm:last_zero}.
\end{rem}
\begin{rem}
We can recover Lamperti's arcsine law \eqref{eq:Bessel_arcsine} by letting $\mu\to 0$. Similarly, letting $t\to\infty$ and appealing to the beta-gamma algebra shows that $g_\infty$ under $\mathbb{P}^{\delta,\mu}_0$ is a $\mathrm{Gamma}(\alpha,\frac{1}{2}\mu^2)$ random variable, cf. \cite[Section 7]{infinite_divis}.
\end{rem}
\begin{comment}
with density 
\[
\mathbb{P}_0^{\delta,\mu}(g_\infty\in\D{t})=\left(\frac{1}{2}\mu^2\right)^\alpha\frac{t^{\alpha-1}}{\Gamma(\alpha)}e^{-\frac{1}{2}\mu^2 t}\D{t},~t>0.
\]
\end{comment}

Theorem \ref{thm:last_zero} has a dual formulation in terms of the first zero after a fixed time which we state below as Theorem \ref{thm:first_zero}. We prove this directly and also show in Section \ref{sec:time_inversion} that either Theorem \ref{thm:last_zero} or Theorem \ref{thm:first_zero} can be deduced from the other using the time inversion property.
\begin{thm}\label{thm:first_zero}
Let $d_t=\inf\{s\geq t:X_s=0\}$ be the first zero after time $t$ of a Bessel process with dimension $\delta\in(0,2)$ and no drift starting at $x>0$. Put $\alpha=1-\frac{\delta}{2}$ and let $A_\alpha$ and $E_x$ be independent $\mathrm{Beta}(\alpha,1-\alpha)$ and $\mathrm{Exp}(\frac{1}{2}x^2)$ random variables, respectively. Then we have the independent factorization
\[
d_t\text{ under }\,\mathbb{P}^\delta_x\,\stackrel{\mathcal{L}}{=}\,\max\left\{t,\frac{1}{E_x}\right\}\,\frac{1}{A_\alpha}.
\]
\end{thm}
\begin{rem}
By letting either $x\to 0$ or $t\to 0$, we see that $d_t$ under $\mathbb{P}^\delta_0$ is a shifted and scaled beta prime random variable while $\tau_0$ under $\mathbb{P}^\delta_x$ is an inverse gamma random variable.
\end{rem}

\section{Preliminary results}\label{sec:preliminaries}

\subsection{Time inversion property and duality}\label{sec:time_inversion}

Our proofs rely on some well-known properties of Bessel processes and Bessel bridges which we now recall, starting with the aforementioned time inversion property \cite[Theorem 2.1]{wide_sense}. For $\delta>0$ and $x,\mu\geq 0$, this states that 
\begin{equation}\label{eq:time_inversion}
\left(tX_\frac{1}{t}:t>0;\,\mathbb{P}_x^{\delta,\mu}\right)\stackrel{\mathcal{L}}{=}\left(X_t:t>0;\,\mathbb{P}_\mu^{\delta,x}\right).
\end{equation}
In other words, time inversion preserves dimension but swaps the drift and starting position of Bessel processes with drift. Since Bessel processes without drift satisfy the usual Brownian scaling property, we get from \eqref{eq:time_inversion} and \eqref{eq:limit_speed} that for any $\delta,c>0$
\begin{equation}\label{eq:drift_scaling}
\left(cX_\frac{t}{c^2}:t\geq 0;\,\mathbb{P}_0^{\delta,\mu}\right)\stackrel{\mathcal{L}}{=}\left(X_t:t\geq 0;\,\mathbb{P}_0^{\delta,\frac{\mu}{c}}\right).
\end{equation}

The time inversion property \eqref{eq:time_inversion} can also be used to establish a duality relation between $g_t$ and $d_t$. More precisely, for any $t>0$ and $x\geq 0$ with $\delta\in (0,2)$ we have
\begin{align}
d_t\text{ under }\,\mathbb{P}^\delta_x\,&\stackrel{\mathcal{L}}{=}\,\inf\left\{s\geq t:s X_\frac{1}{s}=0\right\}\text{ under }\,\mathbb{P}^{\delta,x}_0 \nonumber\\
&=\,\frac{1}{\sup\left\{s\leq\frac{1}{t}:X_s=0\right\}}\text{ under }\,\mathbb{P}^{\delta,x}_0 \nonumber\\
&=\,\frac{1}{g_\frac{1}{t}}\text{ under }\,\mathbb{P}^{\delta,x}_0.\label{eq:dual_times}
\end{align}
Using \eqref{eq:dual_times}, it is easy to deduce Theorem \ref{thm:first_zero} from Theorem \ref{thm:last_zero} and vice versa.

\subsection{Bessel bridges with \texorpdfstring{$\delta>0$}{d>0}}

Next we introduce the notation $\mathbb{P}_{x\to y}^{\delta,T}$ for the law of a Bessel process with dimension $\delta>0$ which starts at $x\geq 0$ and is conditioned to be at $y\geq 0$ at time $T>0$, that is, the law of a Bessel bridge with dimension $\delta>0$ from $x$ to $y$ of length $T$. While the appearance of an arrow in the notation for bridge laws should prevent mistaking the $T$ for drift, we also use $B$ for the coordinate process of a bridge instead of $X$ to further the distinction.

The following lemma is a consequence of the time inversion property \eqref{eq:time_inversion} and shows how Bessel processes with drift and Bessel bridges are related to each other through a space-time transformation.
\begin{lem}\label{lem:bridges}
Suppose $\delta,T>0$ and $x,\mu\geq 0$. Then we have:
\begin{enumerate}[label=\emph{\roman*.}]
\item $\displaystyle \left(B_t:0\leq t<T;\,\mathbb{P}_{0\to \mu}^{\delta,T}\right)\stackrel{\mathcal{L}}{=}\left((T-t)X_\frac{t}{T(T-t)}:0\leq t<T;\,\mathbb{P}_0^{\delta,\mu}\right),$
\item $\displaystyle \left(X_t:t\geq 0;\,\mathbb{P}_0^{\delta,\mu}\right)\stackrel{\mathcal{L}}{=}\left(\frac{1+Tt}{T}B_\frac{T^2 t}{1+Tt}:t\geq 0;\,\mathbb{P}_{0\to \mu}^{\delta,T}\right)$.
\end{enumerate}
\end{lem}

\begin{proof}
We prove part $\mathrm{i.}$ by starting from Theorem 5.8 in \cite{infinite_divis} and using \eqref{eq:drift_scaling}
\begin{align*}
\left(B_t:0\leq t<T;\,\mathbb{P}_{0\to \mu}^{\delta,T}\right)&\stackrel{\mathcal{L}}{=}\left(\frac{T-t}{T}X_\frac{Tt}{T-t}:0\leq t<T;\,\mathbb{P}_0^{\delta,\frac{\mu}{T}}\right)\\
&\stackrel{\mathcal{L}}{=}\left((T-t)X_\frac{t}{T(T-t)}:0\leq t<T;\,\mathbb{P}_0^{\delta,\mu}\right).
\end{align*}

Part $\mathrm{ii.}$ can be deduced from part $\mathrm{i.}$ via
\begin{align*}
&\left(\frac{1+Tt}{T}B_\frac{T^2 t}{1+Tt}:t\geq 0;\,\mathbb{P}_{0\to \mu}^{\delta,T}\right)\\
&~~\stackrel{\mathcal{L}}{=}\left(\frac{1+Tt}{T}\left(T-\frac{T^2 t}{1+Tt}\right)X_\frac{\frac{T^2 t}{1+Tt}}{T\left(T-\frac{T^2 t}{1+Tt}\right)}:t\geq 0;\,\mathbb{P}_0^{\delta,\mu}\right)\\
&~~=\left(X_t:t\geq 0;\,\mathbb{P}_0^{\delta,\mu}\right).
\end{align*}
\end{proof}

\subsection{Bessel bridges with \texorpdfstring{$\delta=0$}{d=0}}

It is well known that $0$ is an absorbing state for the Bessel process of dimension $0$ and this will have implications for the corresponding bridges. Before defining Bessel bridges in this case, we first recall the particularly simple distribution function of the absorption time. More generally, we let $\tau_y=\inf\{t>0:X_t=y\}$ denote the first hitting time of $y\in\mathbb{R}$ by the coordinate process. Then from \cite[Corollary XI.1.4]{Revuz_Yor} we have 
\begin{equation}\label{eq:absorb_dist}
\mathbb{P}_x^0(\tau_0\leq t)=e^{-\frac{x^2}{2t}},~t>0.
\end{equation}
From this it follows that 
\begin{equation}\label{eq:absorption}
\tau_0\text{ under }\mathbb{P}_x^0\,\stackrel{\mathcal{L}}{=}\, \frac{1}{E_x}
\end{equation}
where $E_x$ is an $\mathrm{Exp}(\frac{1}{2}x^2)$ random variable.

Now we expound on the subtlety in the definition of Bessel bridges of dimension $0$ that stems from $0$ being an absorbing state for the underlying unconditioned process, see also \cite[Section 5.3]{Bessel_decomp}. When $x>0$, the bridge law $\mathbb{P}_{x\to 0}^{0,T}$ results from conditioning a $0$-dimensional Bessel path of duration $T$ starting at $x$ to be absorbed before time $T$. Note that this conditioned absorption time will almost surely occur strictly between $0$ and $T$. The bridge law $\mathbb{P}_{0\to x}^{0,T}$ is simply the law of the time reversed bridge under $\mathbb{P}_{x\to 0}^{0,T}$, namely 
\[
\left(B_t:0\leq t\leq T;\,\mathbb{P}_{0\to x}^{0,T}\right)\stackrel{\mathcal{L}}{=}\left(B_{T-t}:0\leq t\leq T;\,\mathbb{P}_{x\to 0}^{0,T}\right).
\]
When both $x,y>0$, the bridge law $\mathbb{P}_{x\to y}^{0,T}$ is defined just as in the $\delta>0$ case. When $x=y=0$, the law $\mathbb{P}_{0\to 0}^{0,T}$ is degenerate and assigns probability $1$ to the constant $0$ path. 

By also conditioning on the exact time of absorption, we can relate a Bessel bridge of dimension $0$ with a bridge of dimension $4$. More precisely, if $x>0$ and $0<S\leq T$, then we have 
\begin{equation*}
\left(B_t:0\leq t\leq S;\,\mathbb{P}_{x\to 0}^{0,T}\middle|\tau_0=S\right)\stackrel{\mathcal{L}}{=}\left(B_t:0\leq t\leq S;\,\mathbb{P}_{x\to 0}^{4,S}\right).
\end{equation*}
From this it follows by time reversal that 
\begin{equation}\label{eq:0_4_bridge}
\left(B_{t+T-S}:0\leq t\leq S;\,\mathbb{P}_{0\to x}^{0,T}\middle|g_T=T-S\right)\stackrel{\mathcal{L}}{=}\left(B_t:0\leq t\leq S;\,\mathbb{P}_{0\to x}^{4,S}\right).
\end{equation}

Lemma \ref{lem:bridges} does not apply when $\delta=0$ so we need another result for this case. The following lemma serves this purpose by using the same space-time transformation to connect a Bessel bridge of dimension $0$ with the waiting Bessel process of dimension $4$ which appears in Theorem \ref{thm:decomp}.
\begin{lem}\label{lem:zero_bridge}
For $\mu>0$, let $E_\mu$ be an independent $\mathrm{Exp}(\frac{1}{2}\mu^2)$ random variable. Then
\[
\left(\left(1+t\right) B_\frac{t}{1+t}:t\geq 0;\,\mathbb{P}_{0\to \mu}^{0,1}\right)\stackrel{\mathcal{L}}{=}\left(X_{(t-E_\mu)^+}:t\geq 0;\,\mathbb{P}_0^{4,\mu}\right).
\]
\end{lem}

\begin{proof}
It follows from \eqref{eq:0_4_bridge} that the bridge $(B_t:0\leq t\leq 1)$ under $\mathbb{P}_{0\to \mu}^{0,1}$ can be split into two independent pieces by conditioning on $g_1$. Indeed, we can sample this bridge by first drawing $g_1$ under $\mathbb{P}_{0\to \mu}^{0,1}$ and then sampling a bridge under $\mathbb{P}_{0\to \mu}^{4,1-g_1}$ that is appended to a constant $0$ path of length $g_1$. More precisely, let $\gamma$ be an independent random variable distributed like $g_1$ under $\mathbb{P}_{0\to \mu}^{0,1}$. Then
\begin{equation*}
\left(B_t:0\leq t\leq 1;\,\mathbb{P}_{0\to \mu}^{0,1}\right)\stackrel{\mathcal{L}}{=}\left(B_{(t-\gamma)^+}:0\leq t\leq 1;\,\mathbb{P}_{0\to \mu}^{4,1-\gamma}\right).
\end{equation*}
Applying the applicable space-time transformation to both sides of this equality in law results in
\begin{equation}\label{eq:transform_split}
\left(\left(1+t\right) B_\frac{t}{1+t}:t\geq 0;\,\mathbb{P}_{0\to \mu}^{0,1}\right)\stackrel{\mathcal{L}}{=}\left((1+t)B_{\left(\frac{t}{1+t}-\gamma\right)^+}:t\geq 0;\,\mathbb{P}_{0\to \mu}^{4,1-\gamma}\right).
\end{equation}

Since $0<\gamma<1$ almost surely, notice that 
\begin{align*}
\left(\frac{t}{1+t}-\gamma\right)^+&=
\begin{cases}
0&0\leq t<\frac{\gamma}{1-\gamma}\\
\frac{t}{1+t}-\gamma&t\geq \frac{\gamma}{1-\gamma}
\end{cases}\\
&=\begin{cases}
0&0\leq t<\frac{\gamma}{1-\gamma}\\
\frac{(1-\gamma)(t-\gamma -\gamma t)}{(1-\gamma)(1+t)}&t\geq \frac{\gamma}{1-\gamma}
\end{cases}\\
&=\frac{(1-\gamma)^2\left(t-\frac{\gamma}{1-\gamma}\right)^+}{1+(1-\gamma)\left(t-\frac{\gamma}{1-\gamma}\right)^+}
\end{align*}
and similar calculations show that 
\[
1+t=\frac{1+(1-\gamma)\left(t-\frac{\gamma}{1-\gamma}\right)^+}{1-\gamma},~t\geq \frac{\gamma}{1-\gamma}.
\]
Writing $f(t)$ for the $(t-\frac{\gamma}{1-\gamma})^+$ which appears in both of these identities, now we can make the appropriate substitutions in the right-hand side of \eqref{eq:transform_split} and then appeal to part $\mathrm{ii.}$ of Lemma \ref{lem:bridges} to yield 
\begin{align}
&\left(\left(1+t\right) B_\frac{t}{1+t}:t\geq 0;\,\mathbb{P}_{0\to \mu}^{0,1}\right)\nonumber\\
&~~\stackrel{\mathcal{L}}{=}\Bigg(\left(\frac{1+(1-\gamma)f(t)}{1-\gamma}\right)B_\frac{(1-\gamma)^2 f(t)}{1+(1-\gamma)f(t)}:t\geq 0;\,\mathbb{P}_{0\to \mu}^{4,1-\gamma}\Bigg)\nonumber\\
&~~\stackrel{\mathcal{L}}{=}\left(X_{\left(t-\frac{\gamma}{1-\gamma}\right)^+}:t\geq 0;\,\mathbb{P}_0^{4,\mu}\right).\label{eq:holding}
\end{align}
 
It remains to identify the distribution of the $\frac{\gamma}{1-\gamma}$ appearing in \eqref{eq:holding}. By hypothesis, this random variable is distributed like $\frac{g_1}{1-g_1}$ under $\mathbb{P}_{0\to \mu}^{0,1}$ and is also independent of the Bessel process that appears in \eqref{eq:holding}. From \eqref{eq:absorb_dist}, it follows that $g_1$ under $\mathbb{P}_{0\to \mu}^{0,1}$ has distribution function 
\begin{align*}
\mathbb{P}_{0\to \mu}^{0,1}(g_1\leq t)&=\mathbb{P}_{\mu\to 0}^{0,1}(1-\tau_0\leq t)\\
&=1-e^{-\frac{\mu^2 t}{2(1-t)}}.
\end{align*}
Hence
\begin{align*}
\mathbb{P}_{0\to \mu}^{0,1}\left(\frac{g_1}{1-g_1}\leq t\right)&=\mathbb{P}_{0\to \mu}^{0,1}\left(g_1\leq\frac{t}{1+t}\right)\\
&=1-e^{-\frac{1}{2}\mu^2 t}.
\end{align*}
From this we deduce that the $\frac{\gamma}{1-\gamma}$ appearing in \eqref{eq:holding} is an independent $\mathrm{Exp}(\frac{1}{2}\mu^2)$ random variable and the proof is complete.
\end{proof}

\subsection{Additivity property}\label{sec:additivity}

Lastly, we recall the \emph{additivity property} of squared Bessel processes \cite[Theorem XI.1.2]{Revuz_Yor}. This property states that if $X$ and $X'$ are independent squared Bessel processes of dimensions $\delta,\delta'\geq 0$ starting from $x,x'\geq 0$, then their sum $(X_t+X'_t:t\geq 0)$ is a squared Bessel process of dimension $\delta+\delta'$ starting from $x+x'$. A more succinct statement of the additivity property is
\begin{equation}\label{eq:additivity}
\mathbb{Q}_x^\delta*\mathbb{Q}_{x'}^{\delta'}=\mathbb{Q}_{x+x'}^{\delta+\delta'}
\end{equation}
where we used $\mathbb{Q}_x^\delta*\mathbb{Q}_{x'}^{\delta'}$ to denote the law of the sum of independent processes with laws $\mathbb{Q}_x^\delta$ and $\mathbb{Q}_{x'}^{\delta'}$. The additivity property also applies to squared Bessel bridges and for $\delta,\delta'\geq 0$ and $x,x'\geq 0$ we have the statements
\begin{equation}\label{eq:bridge_additivity}
\begin{split}
\mathbb{Q}_{x\to 0}^{\delta,T}*\mathbb{Q}_{x'\to 0}^{\delta',T}&=\mathbb{Q}_{x+x'\to 0}^{\delta+\delta',T}\\
\\
\mathbb{Q}_{0\to x}^{\delta,T}*\mathbb{Q}_{0\to x'}^{\delta',T}&=\mathbb{Q}_{0\to x+x'}^{\delta+\delta',T}.
\end{split}
\end{equation}
The analogous result for bridges with general starting and ending points is more complicated, see \cite[Theorem 5.8]{Bessel_decomp}.

\section{Proofs of the main theorems}\label{sec:proofs}

\subsection{Proof of Theorem \ref{thm:decomp}}

\begin{proof}[Proof of Theorem \ref{thm:decomp}]
We apply part $\mathrm{ii.}$ of Lemma \ref{lem:bridges}, the additivity property for bridges \eqref{eq:bridge_additivity}, then part $\mathrm{ii.}$ of Lemma \ref{lem:bridges} and Lemma \ref{lem:zero_bridge} to write
\begin{align*}
&\left(X_t^2:t\geq 0;\,\mathbb{P}_0^{\delta,\mu}\right)\\
&~~\stackrel{\mathcal{L}}{=}\left(\left(1+t\right)^2 B_\frac{t}{1+t}^2:t\geq 0;\,\mathbb{P}_{0\to \mu}^{\delta,1}\right)\\
&~~\stackrel{\mathcal{L}}{=}\left(\left((1+t)B_\frac{t}{1+t}\right)^2+\left((1+t)B_\frac{t}{1+t}'\right)^2:t\geq 0;\,\mathbb{P}_{0\to 0}^{\delta,1}(B)\times\mathbb{P}_{0\to \mu}^{0,1}(B')\right)\\
&~~\stackrel{\mathcal{L}}{=}\left(X_t^2+X_{\left(t-E_\mu\right)^+}'^2:t\geq 0;\,\mathbb{P}_0^\delta(X)\times\mathbb{P}_0^{4,\mu}(X')\right).
\end{align*}
In light of \eqref{eq:Bessel_square}, this proves the theorem.
\end{proof}

\subsection{Proof of Theorem \ref{thm:last_zero}}\label{sec:last_zero}

Arguably, the most direct way to analyze the distribution of $g_t$ under $\mathbb{P}_0^{\delta,\mu}$ is via the absolute continuity relation \eqref{eq:h_transform} using the joint law of $g_t$ and $X_t$ under $\mathbb{P}_0^\delta$. Deriving this joint law is relatively straightforward. The $g_t$ marginal is already known from Lamperti's arcsine law \eqref{eq:Bessel_arcsine} and the conditional law of $X_t$ can be deduced using the bridge-meander path decomposition whereby conditioning on $g_t$, the path $(X_s:0\leq s\leq t)$ under $\mathbb{P}_0^\delta$ splits into a $\delta$-dimensional Bessel bridge of length $g_t$ and an independent $\delta$-dimensional Bessel meander of length $t-g_t$. An explicit joint density follows from the Imhof relation for Bessel meanders \cite[Section 3.6]{aspects} and the transition density for Bessel processes \cite[Section XI.1]{Revuz_Yor}. Now the joint density of $g_t$ and $X_t$ under $\mathbb{P}_0^{\delta,\mu}$ can be read off from \eqref{eq:h_transform}.

However, one major drawback of using the above outlined approach is that onerous computations involving modified Bessel functions are necessary in order to compute the marginal density of $g_t$ under $\mathbb{P}_0^{\delta,\mu}$ and prove that it coincides with that of $\min\{t,E_\mu\}A_\alpha$. Indeed, even in the Brownian motion case where Bessel functions are absent, tedious calculations are required to verify \eqref{eq:ESG} via this approach, see \cite[Section 2]{Iafrate}. In fact, it might even be worthwhile to reverse this line of reasoning and see if Theorem \ref{thm:last_zero} leads to any new integral identities for modified Bessel functions, though we don't pursue this question here. By contrast, the method of proof we employ below completely avoids the computation issue by appealing to Theorem \ref{thm:decomp} instead.

\begin{proof}[Proof of Theorem \ref{thm:last_zero}]
Since $X'$ under $\mathbb{Q}_0^{4,\mu}$ never returns to $0$ once it starts, we can use the additive decomposition from Theorem \ref{thm:decomp} to write
\begin{align}
&\sup\{s\leq t:X_s=0\}\text{ under }\,\mathbb{Q}^{\delta,\mu}_0\nonumber\\
&~~\stackrel{\mathcal{L}}{=}\sup\{s\leq t:X_s+X_{(s-E_\mu)^+}'=0\}\text{ under }\,\mathbb{Q}_0^\delta(X)\times\mathbb{Q}_0^{4,\mu}(X')\nonumber\\
&~~\stackrel{\mathcal{L}}{=}\sup\big\{s\leq\min\{t,\,E_\mu\}:X_s=0\big\}\text{ under }\,\mathbb{Q}_0^\delta.\label{eq:zero_split}
\end{align}
We can use the independence of $E_\mu$ and $X$ along with Bessel scaling to factor out the $\min\{t,\,E_\mu\}$ from inside the $\sup$ appearing in \eqref{eq:zero_split}. Together with the fact that the zeros of a process and its square are the same, this allows us to conclude that 
\begin{align*}
g_t\text{ under }\,\mathbb{P}^{\delta,\mu}_0\,&\stackrel{\mathcal{L}}{=}\,\min\{t,\,E_\mu\}\,\sup\{s\leq 1:X_s=0\}\text{ under }\,\mathbb{Q}_0^\delta\\
&\stackrel{\mathcal{L}}{=}\,\min\{t,E_\mu\}\, g_1\text{ under }\,\mathbb{P}_0^\delta.
\end{align*}
Now the desired result follows from Lamperti's arcsine law \eqref{eq:Bessel_arcsine}.
\end{proof}

\subsection{Proof of Theorem \ref{thm:first_zero}}\label{sec:first_zero}

As mentioned in Section \ref{sec:time_inversion}, Theorem \ref{thm:first_zero} follows from a combination of Theorem \ref{thm:last_zero} and the duality relation \eqref{eq:dual_times}. However, here we opt for a more direct proof that uses the additivity property for squared Bessel processes without drift from Section \ref{sec:additivity}.

\begin{proof}[Proof of Theorem \ref{thm:first_zero}]
Since $0$ is an absorbing state for $X'$ under $\mathbb{Q}_{x^2}^\delta$, we can use the additivity property \eqref{eq:additivity} to write
\begin{align}
&\inf\{s\geq t:X_s=0\}\text{ under }\,\mathbb{Q}^\delta_{x^2}\nonumber\\
&~~\stackrel{\mathcal{L}}{=}\inf\{s\geq t:X_s+X_s'=0\}\text{ under }\,\mathbb{Q}_0^\delta(X)\times\mathbb{Q}_{x^2}^0(X')\nonumber\\
&~~\stackrel{\mathcal{L}}{=}\inf\big\{s\geq\max\{t,\,\tau_0'\}:X_s=0\big\}\text{ under }\,\mathbb{Q}_0^\delta(X)\times\mathbb{Q}_{x^2}^0(X').\label{eq:first_split}
\end{align}
The independence of $\tau_0'$ and $X$ allows us to apply Bessel scaling to \eqref{eq:first_split}, thereby factoring out the $\max\{t,\,\tau_0'\}$ from inside the $\inf$. In conjunction with \eqref{eq:Bessel_square} and \eqref{eq:absorption}, this implies that
\begin{align}
d_t\text{ under }\,\mathbb{P}^\delta_x\,&\stackrel{\mathcal{L}}{=}\,\max\{t,\,\tau_0'\}\,\inf\{s\geq 1:X_s=0\}\text{ under }\,\mathbb{Q}_0^\delta(X)\times\mathbb{Q}_{x^2}^0(X')\nonumber\\
&\stackrel{\mathcal{L}}{=}\,\max\left\{t,\,\frac{1}{E_x}\right\}\,d_1\text{ under }\,\mathbb{P}_0^\delta.\label{eq:first_law}
\end{align}
Now we can use \eqref{eq:dual_times} and \eqref{eq:Bessel_arcsine} to rewrite \eqref{eq:first_law} as
\[
d_t\text{ under }\,\mathbb{P}^\delta_x\,\stackrel{\mathcal{L}}{=}\,\max\left\{t,\frac{1}{E_x}\right\}\,\frac{1}{A_\alpha}
\]
which completes the proof.
\end{proof}

\begin{acknowledgments}
The author would like to thank Jim Pitman for providing useful comments on an earlier draft.
\end{acknowledgments}

\bibliography{last_zero}
\bibliographystyle{alphabbrev}

\end{document}